\documentclass{amsart}
\usepackage{latexsym}
\usepackage{amsfonts}
\usepackage{amssymb}
\usepackage{enumerate}
\usepackage{amssymb,latexsym,amsxtra,amscd}
\usepackage{amsmath,amsthm,amsfonts, amssymb,amscd}
\usepackage[all]{xy}

\newtheorem{theorem}{Theorem}[section]
\newtheorem{lemma}[theorem]{Lemma}%[section]
\newtheorem{remark}[theorem]{Remark}%[section]
\newtheorem{corollary}[theorem]{Corollary}%[section]
\newtheorem{proposition}[theorem]{Proposition}

\newcommand{\squig}{\rightsquigarrow}

\long\def\alert#1{\smallskip{\hskip\parindent\vrule%
\vbox{\advance\hsize-2\parindent\hrule\smallskip\parindent.4\parindent%
\narrower\noindent#1\smallskip\hrule}\vrule\hfill}\smallskip}
\newcommand{\Val}{{\rm Val}}

\begin{document}
\title[On Normal-Valued Basic Pseudo Hoops]{On Normal-Valued  Basic Pseudo Hoops}
\author[M. Botur, A. Dvure\v{c}enskij,  and T. Kowalski]{Michal Botur$^1$, Anatolij Dvure\v{c}enskij$^2$,  and Tomasz Kowalski$^3$}
\date{}
\maketitle

\begin{center}  \footnote{Keywords:
Pseudo MV-algebra, pseudo BL-algebra, pseudo hoop, filter, Holland's
Representation Theorem, normal-valued pseudo hoop, base of
equations.

AMS classification:  06D35, 03G12,
03B50, 81P10.

MB thanks for the support by SAIA, Slovakia, and by MSM 6198959214 of the RDC of the Czech Government, and by GA\v{C}R P201/11/P346, Czech Republic,   AD thanks  for the
support by Center of Excellence SAS -~Quantum Technologies~-,  ERDF
OP R\&D Projects CE QUTE ITMS 26240120009 and meta-QUTE ITMS
26240120022, the grant VEGA No. 2/0032/09 SAV. }
\small{Department of Algebra and Geometry\\  Faculty of Natural Sciences, Palack\'y University\\ 17. listopadu 12, CZ-771 46 Olomouc, Czech Republik\\
$^2$ Mathematical Institute,  Slovak Academy of Sciences\\
\v Stef\'anikova 49, SK-814 73 Bratislava, Slovakia\\
$^3$ Department of Mathematics and Statistics\\
University of Melbourne\\ Parkville, VIC 3010, Australia\\
E-mail: {\tt botur@inf.upol.cz},\ {\tt dvurecen@mat.savba.sk},  \\ {\tt kowatomasz@gmail.com}}
\end{center}

\begin{abstract}  We show that every pseudo hoop satisfies the Riesz
Decomposition Property. We visualize basic pseudo hoops by functions
on a linearly ordered set.  Finally, we study normal-valued basic
pseudo hoops giving a countable base of equations for them.
\end{abstract}

\section{Introduction}%1

The Romanian algebraic school during the last decade contributed a
lot to  noncommutative generalizations of many-valued reasoning
which generalizes MV-algebras by C.C. Chang \cite{Cha}. They
introduced pseudo MV-algebras, \cite{GeIo} (independently introduced
also in \cite{Rac} as generalized MV-algebras), pseudo BL-algebras,
\cite{DGI1, DGI2}, pseudo hoops, \cite{GLP}. We recall that pseudo
BL-algebras are also a noncommutative generalization of P. H\'ajek's
BL-algebras: a variety that is an algebraic counterpart of fuzzy
logic, \cite{Haj}.

However, as it was recently recognized,  many of these notions have
a very close connections with notions introduced already by B.
Bosbach in his pioneering papers on various classes of semigroups:
among others he introduced complementary semigroups (today known as
pseudo-hoops). A deep investigation of these structures can be found
in his papers \cite{Bos1, Bos2}; more information are available in
his recent papers \cite{Bos3,Bos4}. Nowadays, all these structures
can be also studied under one common roof, as residuated lattices,
\cite{GaTs}.

Now all these structures are  intensively studied by many experts.
Very important results were presented in \cite{JiMo}. In the paper
\cite{Dvu4}, it was proved that every linearly ordered pseudo hoop
is an ordinal sum of negative cones or intervals of lattice-ordered
groups, see also \cite{AgMo}. The paper \cite{DGK} introduced
interesting classes of pseudo hoops, like systems $\mathcal{MPH}$
and $\mathcal{MPH}_b$  of all pseudo hoops  (bounded pseudo hoops)
$M$ such that every maximal filter of $M$ is normal, and the system
$\mathcal{NVPH}$ of normal-valued basic pseudo-hoops $M$ such that
every value in $M$ is normal in its cover. The latter one is
inspired by analogous notions from theory of $\ell$-groups. In
\cite{DGK}, there was proved that $\mathcal{NVPH} \subset
\mathcal{MPH},$ $\mathcal{MPH}_b \subset \mathcal{MPH}$ and
$\mathcal{NVPH},$ $\mathcal{MPH}_b$ are varieties but
$\mathcal{MPH}$ is not a variety, \cite[Rem 4.2]{DGK}.

The main aim is to continue in the study of pseudo hoops, focusing
on  normal-valued ones. We present an equational basis  of
normal-valued basic pseudo hoops. In addition, we show that every
pseudo hoop satisfies the Riesz Decomposition Property (RDP) and we
present also a  Holland's type representation of basic pseudo hoops.

The paper is organized as follows.  Section 2 gathers the basic
notions and properties of pseudo hoops and Section 3 deals with
basic pseudo hoops. Section 4 proves the Riesz Decomposition
Property for pseudo hoops, and presents some results on filters.
Some kind of the Holland Representation Theorem for basic pseudo
hoops which enables us to visualize them by functions on a linearly
ordered set is presented in Section 5. Finally, Section 6 studies
normal-valued basic pseudo hoops and presents a countable base of
equations characterizing them. In addition two open questions are
formulated.

\section{Basic Facts and Properties} %2

We recall that according to \cite{GLP}, a \textit{pseudo hoop} is
an algebra $(M; \odot, \to,\squig,1)$ of type $\langle 2,2,2,0
\rangle$ such that, for all $x,y,z \in M,$

\begin{enumerate}

\item[{\rm (i)}]   $x\odot 1 = x = 1 \odot x;$

 \item[{\rm (ii)}] $x\to x = 1 = x\squig x;$

\item[{\rm (iii)}] $(x\odot y) \to z = x \to (y\to z);$

 \item[{\rm (iv)}] $(x \odot y) \squig z = y \squig
(x\squig z);$

 \item[{\rm (v)}] $(x\to y) \odot x= (y\to x)\odot y =
x\odot (x\squig y) = y \odot (y \squig x).$

\end{enumerate}

We recall that  $\odot$ have higher priority than $\to$ or $\squig$,
and those higher than $\wedge$ and $\vee,$ and $\wedge$ is higher
than $\vee.$

If $\odot$ is commutative (equivalently $ \to = \squig$), $M$ is
said to be a \textit{hoop}.  If we set $x \le y$ iff $x \to y=1$
(this is equivalent to $x \squig y =1$), then $\le$ is a partial
order such that $x\wedge y = (x\to y)\odot x$ and $M$ is a
$\wedge$-semilattice.

We say that a pseudo hoop $M$

\begin{enumerate}

\item[(i)] is  {\it bounded} if there is a least element $0,$
otherwise, $M$ is {\it unbounded},

\item[(ii)] satisfies \textit{prelinearity} if, given $x,y \in M,$
$(x\to y)\vee (y\to x)$ and $(x\squig y)\vee (y\squig x)$ are
defined in $M$ and they are equal $1,$

\item[(iii)] is \textit{cancellative} if $x\odot y=x\odot z$ and $s\odot x= t
\odot x$ imply $y= z$ and $s= t,$

\item[(iv)] is a \textit{pseudo BL-algebra}  if $M$ is a bounded
lattice satisfying  prelinearity.
\end{enumerate}

For a pseudo BL-algebra, we define $x^-=x\to 0$ and $x^\sim =
x\squig 0.$ A pseudo BL-algebra is said to be a pseudo MV-algebra if
$x^{-\sim}=x=x^{\sim -}$ for every $x \in M.$

From (v) of the definition of pseudo hoops we have that a pseudo
hoop is cancellative iff $x\odot y\le x\odot z$ and $s\odot x\le t
\odot x$ imply $y\le z$ and $s\le t.$

Many examples of pseudo hoops can be made from $\ell$-groups. Now
let $G$ be an $\ell$-group (written multiplicatively and with a
neutral element $e$). On the negative cone $G^-=\{g\in G:\ g\le e\}$
we define: $x\odot y :=xy,$ $x\to y := (yx^{-1})\wedge e,$ $x\squig
y :=(x^{-1}y)\wedge e,$ for $x,y \in G^-.$ Then
$(G^-;\odot,\to,\squig,e)$ is an unbounded (whenever $G\ne \{e\}$)
cancellative pseudo hoop. Conversely, according to \cite[Prop
5.7]{GLP}, every cancellative pseudo hoop is isomorphic to some
$(G^-;\odot,\to,\squig,e).$

If $u\ge e$ is a strong unit unit (= order unit) in $G$, we define
on $[-u,e]$ operations $x\odot y :=(xy)\vee(-u),$ $x\to y :=
(yx^{-1})\wedge e,$ $x\squig y :=(x^{-1}y)\wedge e,$ for $x,y \in
[-u,e].$ Then $([-u,e];\odot,\to, \squig, -u,e)$ is a bounded pseudo
hoop (= pseudo MV-algebra). By \cite{Dvu1}, every pseudo MV-algebra
is of the form $([-u,e];\odot,\to, \squig, -u,e).$

For any $x \in M$ and any integer $n\ge 0$ we define $x^n$
inductively: $x^0 := 1$ and $x^n := x^{n-1}\odot x$ for $n \ge 1.$

A subset $F$ of a pseudo hoop is said to be a {\it filter} if (i)
$x,y \in F$ implies $x\odot y \in F,$ and (ii) $x\le y$ and $x \in
F$ imply $y \in F.$  We denote by $\mathcal F(M)$ the set of all
filters of $M.$  According to  \cite[Prop 3.1]{GLP},  a subset $F$
is a filter iff (i) $1 \in F$, and (ii) $x, x\to y \in F$ implies
$y\in F$ ($x, x\squig y \in F$ implies $y\in F$), i.e., $F$ is a
{\it deductive system}.  If $a\in M,$ then the filter, $F(a),$
generated by $a$ is the set
$$
F(a)=\{x\in M: x\ge a^n\ \mbox{for some}\ n \ge 1\}.
$$

A filter $F$ is normal if $x\to y \in F$ iff $x\squig y \in F$. This
is equivalent $a\odot F = F\odot a$ for any $a \in M$;  here $a\odot
F = \{a\odot h:\ h \in F\}$ and $F\odot a = \{h\odot a:\ h \in F\}.$
If $F$ is a normal filter, we define $x\theta_F y $ iff $x\to y \in
F$ and $y\to x\in F$, then $\theta_F$ is a congruence on $M$,
\cite[Prop 3.13]{GLP}, and $M/F =\{x/\theta_F:\ x \in M\} $ is again
a pseudo hoop, where $x/\theta_F$ is an equivalence class
corresponding to the element $x \in M,$ we write also $x/F
=x/\theta_F.$ Moreover, there is a one-to-one correspondence,
\cite[Prop 3.15]{GLP}, among the set of normal filters, $F,$ and the
set of congruences.

We recall that a filter $F$ of a pseudo hoop $M$ is called  {\it
maximal} if it is a proper subset of $M$ and not  properly contained
in any proper filter of $M$. We recall that if $M$ is not bounded,
then it can happen that $M$ has no maximal filter; for example this
is true for the real interval $(0,1]$ equipped with $s\odot t =
\min\{s,t\}$, and $s\to t = 1$ iff $s\le t$, otherwise $s\to t = t$
$(s,t \in (0,1])$. In \cite{Dvu3}, it was proved that every linear
pseudo BL-algebra admits a unique maximal filter, and this filter is
normal.

\section{Basic Pseudo Hoops}%3

A pseudo hoop $M$ is said to be {\it basic} if, for all $x,y,z \in
M,$

\begin{enumerate}

\item[{\rm (B1)}] $(x\to y) \to z \le ((y\to x)\to z)\to z$;

 \item[{\rm (B2)}] $(x\squig y) \squig z \le ((y\squig
x)\squig z)\squig z$.

\end{enumerate}

It is straightforward to verify that any linearly ordered pseudo
hoop and hence any representable pseudo hoop (= a subdirect product
of linearly ordered pseudo hoops)  is basic.

By  \cite[Prop 4.6]{GLP},  every basic pseudo hoop is a distributive
lattice. By  \cite[Prop 4.6]{GLP}, $M$ is a distributive lattice
with prelinearity.

We note, see \cite[Lem 2.6]{GLP}, that if $\bigvee_i b_i$ exists,
then so do $ \bigvee_i (a\odot  b_i)$ and $ \bigvee_i (b_i\odot a),$
moreover, $a\odot (\bigvee_i b_i) = \bigvee_i (a\odot b_i)$ and
$(\bigvee_i b_i)\odot a = \bigvee_i (b_i \odot a).$

\begin{proposition}\label{pr:2.1}
If a pseudo hoop $M$ satisfies prelinearity, then $\odot$
distributes $\wedge$ from both sides, i.e. for all $x,y,z \in M,$ we
have

\begin{enumerate}

\item[{\rm (i)}] $z\odot (x\wedge y) = (z\odot x)\wedge (z\odot y),$

\item[{\rm (ii)}] $(x\wedge y)\odot z = (z\odot
z)\wedge (y\odot z).$
\end{enumerate}

\end{proposition}

\begin{proof}
First of all, if $a\le b$, then $a \le c\squig b$ and $a \le c\to b$
for any $c \in M.$ Indeed, $a \le b \le c\squig b.$

Second, for all $a,b,c \in M,$ $(a\squig b)\squig (a\squig
c)=(b\squig a)\squig (b\squig c)$ and $(a\to b)\to (a\to c)=(b\to
a)\to (b\to c).$ In fact, by \cite[Thm 2.2]{GLP}, $(a\squig b)\squig
(a\squig c)=(a\odot (a\squig b))\squig c = (a\wedge b)\squig c =
(b\wedge a)\squig c= (b\odot (b \squig a)) \squig c = (b\squig
a)\squig (b \squig c).$ In the same way we prove the second
equality.

By \cite[Lem 2.5(19)]{GLP}, we have $x \squig y = x\squig (x\wedge
y) \le z \odot x \squig z\odot (x\wedge y).$  Hence, by the first
part, $x\squig y \le (z\odot x \squig z\odot y)\squig (z\odot x
\squig z\odot (x\wedge y)).$ In a similar way, $y\squig x \le
(z\odot y \squig z\odot x)\squig (z\odot y \squig z\odot (x\wedge
y)).$  By the second remark of the proof, the right-hand sides of
the last two inequalities are the same, we denote it by $s.$ Hence,
$x\squig y, y \squig x \le s$ and  prelinearity implies $s=1.$
Therefore, $z \odot x \squig z \odot y \le z\odot x \squig z\odot
(x\wedge y)$ and $(z\odot x)\odot (z\odot x \squig z\odot y) \le z
\odot (x\wedge y),$ i.e., $(z\odot x)\wedge (z\odot y) \le z \odot
(x\wedge y).$ The converse inequality, $z \odot (x\wedge y)\le
(z\odot x)\wedge (z\odot y)$ is obvious.  Hence, (i)  holds.

The proof of (ii) is similar.
\end{proof}

\vspace{2mm}  According to \cite{GLP}, we define, for all $x,y \in
M:$
\begin{align*}
x\vee_1 y &:=((x\squig y) \to y)\wedge ((y \squig x)\to x),\\
x\vee_2 y &:= ((x\to y) \squig y)\wedge ((y \to x)\squig x).
\end{align*}
Then $x,y \le x\vee_iy$ for $i=1,2.$

\begin{proposition}\label{pr:2.2}
If $M$ is a pseudo hoop with prelinearity, then $M$ is  basic,  $M$
is a lattice, and
$$((x\squig y) \to y)\wedge ((y \squig x)\to x)= x\vee y =
((x\to y) \squig y)\wedge ((y \to x)\squig x)\eqno(3.1)
$$
for all $x,y \in M.$
\end{proposition}

\begin{proof}
Since every pseudo hoop is a $\wedge$-semilattice, we have to show
that $x\vee y$ exists in $M.$ Let $a$ be the left-hand side of
(3.1). Due to \cite[Prop 2.11]{GLP}, $a\ge x,y.$  Now let $x,y \le
c.$ We have $a = a \odot 1 = a \odot ((x\squig y)\vee (y \squig x))
= (a\odot (x\squig y))\vee (a \odot (y \squig x)).$   On the other
hand, $a\odot (x\squig y) = [((x\squig y) \to y)\wedge ((y \squig
x)\to x)]\odot (x\squig y)\le ((x\squig y)\to y) \odot (x\squig y) =
(x\squig y) \wedge y \le y\le c.$ In a similar way, we have $a \odot
(y \squig x) \le x\le c.$ Hence, $a \le c.$

The second equality can be proved in a similar approach.

Now applying  \cite[Prop 4.7]{GLP}, we have that $M$ is basic.
\end{proof}

\begin{remark}\label{re:2.3} {\rm
Proposition \ref{pr:2.2} generalizes \cite[Prop 4.7]{GLP} where it
was proved that a pseudo hoop $M$  is basic iff $\vee_1$ and
$\vee_2$ are associative and $(x\squig y)\vee_1 (y\squig x)=1$ for
all $x,y \in M.$ }
\end{remark}

\begin{proposition}\label{pr:2.4}
The variety of bounded pseudo hoops with prelinearity  is termwise
equivalent to the variety of pseudo BL-algebras.
\end{proposition}

\begin{proof}
If $M$ is a bounded pseudo hoop with prelinearity, according to
Proposition \ref{pr:2.2}, $M$ is basic and  due to  \cite[Prop
4.10]{GLP}, $M$ is termwise equivalent to a pseudo BL-algebra.

Now let $M$ be a pseudo BL-algebra, then it is a bounded pseudo hoop
with prelinearity.
\end{proof}

\section{Filters, Prime Filters and the Riesz Decomposition Property}%4

In this section, we extend some results on filters and we show that
every pseudo hoop satisfies the Riesz Decomposition Property. This
property was known only for pseudo MV-algebras, \cite{Dvu1}.

We are saying that a pseudo hoop $M$ satisfies the \textit{Riesz
decomposition property} ((RDP) for short) if $a \ge b\odot c$
implies that there are two elements $b_1\ge b$ and $c_1 \ge c$ such
that $a = b_1\odot c_1.$  For example, (i) every pseudo MV-algebra
satisfies (RDP), (ii) every cancellative pseudo hoop ($\cong G^-$
for some $\ell$-group $G$) satisfies (RDP), (iii) if $M_0$ and $M_1$
satisfies (RDP), so does $M_0\oplus M_1$, (iv) every linearly
ordered pseudo hoop (thanks to the Aglian\`o-Montagna decomposition
of linearly ordered pseudo hoops \cite{Dvu4}) satisfies (RDP), (v)
if $G$ is an $\ell$-group, then the kite pseudo BL-algebra $G^\dag$
satisfies (RDP) (for kites see e.g. \cite{JiMo, DGK}).  In what
follows, we show that all the latter examples are special cases of a
more general result saying that every pseudo hoop satisfies (RDP).

\begin{theorem}\label{th:3.1}
Every pseudo hoop $M$ satisfies {\rm (RDP)}.
\end{theorem}

\begin{proof}
Let $a,b,c\in M$ be such that $b\odot c\leq a$. Then we denote
$$b' := ((c\rightarrow a)\rightsquigarrow a)\rightarrow a, \quad
c' := (c\rightarrow a)\rightsquigarrow a. $$

Clearly $c\leq (c\rightarrow a)\rightsquigarrow a = c'$. Moreover,
$b\odot c\leq a$ yields $b\leq c\rightarrow a$. Thus also
$(c\rightarrow a)\rightsquigarrow a\leq b\rightsquigarrow a$ holds.
Because  pseudo hoops are  residuated structures,
$b\odot((c\rightarrow a)\rightsquigarrow a)\leq a$ and $ b\leq
((c\rightarrow a)\rightsquigarrow a)\rightarrow a = b'$ holds.
Finally, we have

  \begin{eqnarray*}
            b'\odot c' &=& (((c\rightarrow a)\rightsquigarrow a)\rightarrow a)\odot ((c\rightarrow a)\rightsquigarrow a)\\
                &=& ((c\rightarrow a)\rightsquigarrow a) \wedge a\\
                &=& a.
        \end{eqnarray*}
%\end{proof}
\end{proof}

If $M$ is a pseudo hoop and $a,b \in M,$ then
$$
F(a\odot b) =F(a) \vee F(b) = F(b\odot a),\eqno(4.1)
$$
If $a\vee b $ exists in $M$, then, \cite[Prop 3.4]{GLP},
$$
F(a\vee b) = F(a)\cap F(b).\eqno(4.2)
$$

Let $F$ be a filter of a pseudo hoop $M.$  We say that two elements
$a,b \in M$ are in a relation $a \cong_F b$ iff $a\to b, b\to a \in
F.$  Due to \cite[Prop 3.6]{GLP}, $\cong_F$ is an  equivalence
relation. Moreover, $a\cong_F b$ iff $x\odot a =y\odot b$ for some
$x,y \in F.$  We denote by $Fa:=a/F$ the equivalent class
corresponding to the element $a\in M$ with respect to $\cong_F,$
hence  $ F\odot a =\{x\odot a: x \in F\}\subseteq Fa$ and $F\odot 1=
F1=F.$ We can introduce a partial binary operation $\le:=\le_F$ on
$M/F=\{Fa: a\in M\}$ via $Fa \le Fb$ iff $a\to b \in F.$ This is
equivalent to $x\odot a \le b$ for some $x \in F.$ Indeed, let $Fa
\le Fb$, set $x = a\to b \in F$ and then $a\wedge b =(a\to b)\odot a
\le b.$ Conversely, let $x\odot a\le b$ for some $x \in F.$ Then $1=
x\odot a \to b = x \to (a\to b)$ which yields $x\le a\to b$ so that
$a\to b\in F.$

Hence, the relation $\le:=\le_F$ is a partial ordering on the set of
$M/F:$ (i)  clearly $Fa \le Fa,$ (ii) if $Fa\le  Fb$ and $Fb\le Fa,$
then $Fa = Fb,$ and if $Fa\le Fb,$ $Fb\le Fc,$ then $Fa \le Fc$
because we have $v_1\odot a\le b$ and $v_2\odot b\le c$ for some
$v_1,v_2 \in F.$ Then $v_2\odot v_1\odot a \le v_2\odot b \le c.$

These quotient classes are so-called  the \textit{right classes}.
We can define also the left classes under the  equivalence relation
$a\, _F\cong b$ iff  $a\squig b, b\squig a\in F,$  and let $aF$ be
the equivalence class with respect to $\, _F\cong.$  Then $aF \le
bF$ iff $a\odot f\le b$ for some $f \in F.$

Let $\mathcal F(M)$ be the system of all  filters   of a pseudo hoop
$M.$

\begin{proposition}\label{pr:p1}
The system of all filters, $\mathcal F(M),$ of a pseudo hoop $M$ is
a distributive lattice under the set-theoretical inclusion.  In
addition, $F\cap \bigvee_i F_i = \bigvee_i(F\cap F_i).$
\end{proposition}

\begin{proof}
If $\{F_i\}$ is a system of filters, then $\bigvee_i F_i = \{x\in M: x \ge f_1\odot \cdots \odot f_n,\ f_1\in
F_{i_1},\ldots, f_n \in F_{i_n}, $ for some $i_1,\ldots, i_n, n\ge
1\}\in \mathcal F(M),$ and $\bigcap_i F_i \in \mathcal F(M).$

It is clear that $ F\cap \bigvee_i F_i \supseteq \bigvee_i(F\cap
F_i).$  Let $x \in  F\cap \bigvee_i F_i.$  Then $x \ge
f_1\odot\cdots \odot f_n$ where $f_1\in F_{i_1},\ldots,f_n\in
F_{i_n}.$ Because every pseudo hoop satisfies (RDP), Theorem
\ref{th:3.1}, $x= f_1^0\odot \cdots \odot f_n^0$ where $f_j^0\ge
f_j.$ Therefore, $x \le f_j^0$ so that $f_j^0 \in F\cap F_{i_j}$ and
$x \in \bigvee_{j=1}^n (F\cap F_{i_j}) \subseteq \bigvee_i (F\cap
F_i).$

The lattice distributivity is clear from the first part of the present proof.
\end{proof}

\vspace{2mm} A filter $F$ of a pseudo hoop $M$ is said to be
\textit{prime} if, for two filters $F_1,F_2$ on $M,$ $F_1\cap F_2
\subseteq F$ entails $F_1\subseteq F$ or $F_2 \subseteq F.$  We
denote by $\mathcal P(M)$ the system of all prime filters of a
pseudo hoop $M.$

We note a prime filter $F$ is \textit{minimal prime} if it does not
contains properly another prime filter of $M.$ We stress that a
minimal prime filter exists always in any basic pseudo hoop $M$
which admits a maximal lattice ideal of the lattice reduct of $M.$

\begin{proposition}\label{pr:p2}
Let $F$ be a filter of   a basic pseudo hoop $M.$    Let us define the
following statements:

\begin{enumerate}

\item[{\rm (i)}] $F$ is prime.

\item[{\rm (ii)}] If $f\vee g = 1,$ then $f \in F$ or $g \in F.$

\item[{\rm (iii)}] For all $f,g \in M,$ $f\to g \in F$
or $g\to f \in F.$

\item[{\rm (iii')}] For all $f,g \in M,$ $f\squig g \in
F$ or $g\squig f \in F.$

\item[{\rm (iv)}] If $f\vee g \in F,$ then $f \in F$ or
$g \in F.$

\item[{\rm (v)}] If $f,g \in M,$ then there is $c \in
F$ such that $c\odot f \le g$ or $c\odot g \le f.$

\item[{\rm (vi)}] If $F_1$ and $F_2$ are two filters of
$M$ containing $F,$ then $F_1\subseteq F_2$ or $F_2\subseteq F_1.$

\item[{\rm (vii)}] If $F_1$ and $F_2$ are two filters
of $M$ such that $F\subsetneq F_1$ and $F\subsetneq F_2,$ then $F
\subsetneq F_1\cap F_2.$

\item[{\rm (viii)}] If $f,g \notin F,$  then $f\vee g
\notin F.$
\end{enumerate}

Then all statements {\rm (i)}--{\rm (viii)}  are equivalent.

\end{proposition}

\begin{proof} (i) $\Rightarrow$ (ii). By  (4.2), $F(f)
\cap F(g) = F(f\vee g)=F(1)=\{1\},$ so that $F(f) \subseteq F$ or
$F(g)\subseteq F,$ and whence $f \in F$ or $g \in G.$

(ii) $\Rightarrow$ (iii), and (ii) $\Rightarrow$ (iii').  They
follow from prelinearity.

(iii) $\Rightarrow$ (iv). Let $f\vee g \in F.$  Let  $f\to g\in F$
or $g\to f\in F.$ Since $(f\vee g)\to g = f\to g,$ in the first case
we have $g = g \wedge (f\vee g) =((f\vee g)\to g)\odot (f\vee g) \in
F$ and similarly in the second one.  In the same manner, we have
(iii') $\Rightarrow$ (iv).

(iv) $\Rightarrow$ (v).  From prelinearity,   let e.g. $c:=f\to g
\in F.$ Then $c\odot f =(f\to g)\odot f = f\wedge g \le g.$

(v) $\Rightarrow$ (i).  Let $F_1\cap F_2 \subseteq F$ and let
$F_1\subsetneq F$ and $F_2 \subsetneq F.$  There are $f \in
F_1\setminus F$ and $g \in F_2\setminus F.$  By (v), there is $c\in
F$ such that, say $c\odot f \le g.$ By (4.2), we have $F(f\vee
g)=F(f)\cap F(g) \subseteq F_1\cap F_2 \subseteq F$ so that $f\vee
g\in F.$  Therefore, $F \ni c\odot (f\vee g) = c\odot f \vee c\odot
g \le g \in F,$ a contradiction.

(v) $\Rightarrow$ (vi).  Suppose that $f \in F_1 \setminus F_2$ and
$g\in F_2 \setminus F_1.$ Then there is $c \in F$ such that e.g.
$c\odot f \le g$ giving a contradiction $g \in F_1.$

(vi) $\Rightarrow$ (vii). Due to the assumption, $F_1\subseteq F_2$
or $F_2\subseteq F_1$ thus $F\subsetneq F_1\cap F_2.$

Because every pseudo hoop satisfies (RDP), we have the following
implications.

(vii) $\Rightarrow$ (viii). By Proposition \ref{pr:p1} and (4.2), we
have $F \subsetneq (F\vee F(f)) \cap (F\vee F(g)) = F \vee F(f\vee
g)$ giving $f\vee g \notin F.$

(viii) $\Rightarrow$ (iv). This is evident. \end{proof}

Now we present the Prime Filter Theorem for basic pseudo hoops.

\begin{lemma}\label{le:p3}  Let $M$ be a basic pseudo hoop. If $A$
is a lattice ideal of $M$ and $F$ is a filter of $M$ such that
$F\cap A=\emptyset,$  then there is a prime filter $P$ of $M$
containing $F$ and disjoint with $A.$
\end{lemma}

\begin{proof}
According to Zorn's Lemma, there is a maximal filter of $M$
containing $F$ and disjoint with $A.$  Applying criterion
Proposition \ref{pr:p2}(iii), we show that $P$ is prime. If not,
there are two elements $f$ and $g$ such that $f\to g, g\to f\notin
P.$

Let $P_1=P\vee F(f\to g)$ and $P_2 =P\vee F(g\to f).$ Due to the
choice of $P$, there are $c_1 \in P_1 \cap A$ and $c_2 \in P_2 \cap
A.$ Hence, $c_1 \ge \prod_{i=1}^n (s_i\odot (f\to g))$ and $c_2 \ge
\prod_{i=1}^n (t_i\odot (g\to f)),$ where $s_i, t_j \in P.$

Set $s = s_1\odot\cdots \odot s_n,$  $t=t_1\odot \cdots \odot t_n,$
and $u = s\odot t\in P.$

 We recall  an easy equality
$g\vee (h\odot k) \ge (g\vee h)\odot (g\vee k).$

Then $c_1 \vee c_2 \ge \prod_{i=1}^n (s_i\odot (f\to g)) \vee
\prod_{i=1}^n (t_i\odot (g\to f))  \ge \prod_i (\prod_j (u\odot
(f\to g)) \vee (u\odot (g\to f))) \ge \prod_{i,j} (u\odot (f\to g)
\vee u\odot (g\to f))= u^{2n}\in P$ Hence, $c_1\vee c_2 \in P$ that
gives a contradiction.
\end{proof}

We recall that an element $u$ of $M$ is said to be a \textit{strong unit} in $M$
if the filter of $M$ generated by $u$ is equal to $M.$

\begin{remark}\label{re:p3}
{\rm Let $M$ be a basic pseudo hoop.

(1)  The \textit{value} of an element $g\in M\setminus\{1\}$ is any
filter $V$ of $M$ that is maximal with respect to the property
$g\notin V.$ Due to Lemma \ref{le:p3}, a value $V$ exists and it is
prime. Let $\Val(g)$ be the set of all values of $g<1.$  The filter
$V^*$ generated by a value $V$ of $g$ and by the element $g$ is said
to be the \textit{cover} of $V.$

(2) We recall that a filter $F$ is \textit{finitely
meet-irreducible} if, for each two filters $F_1,F_2$ such that $F
\subsetneq F_1$ and $F\subsetneq F_2,$ we have $F \subsetneq F_1\cap
F_2.$ Due to Proposition \ref{pr:p2}(vii), the finite
meet-irreducibility is a sufficient and  necessary condition for a filter $F$ to be  prime.

(3)  Proposition \ref{pr:p2}(iii) says that $F$ is prime iff the set
of quotient classes $\{Fa: a \in M\}$ is linearly ordered.

(4)  Proposition \ref{pr:p2}(vi) says that the system of prime
filters, $\mathcal P(M),$ is a \textit{root system}.

(5) $M$ has a maximal filter iff $M$ admits a strong unit $u.$}
\end{remark}

Importance of values can be seen from the following
characterization.

\begin{lemma}\label{le:p4}
Let $M$ be a basic pseudo hoop. Then $f\le g$ if and only if $Vf \le
Vg$ for all values $V$ in $M.$ Moreover, let given $a \in M\setminus
\{1\},$ $V_a$ be a fixed value of $a.$  Then $f\le g$ if and only if
$V_af \le V_ag$ for each $a\in M.$
\end{lemma}

\begin{proof}
First we show that given a value $V,$ we have
$V (f\wedge g) = V f\wedge V g$ and if $f\vee g$ exists in $M$ then
$V(f\vee g) = Vf\vee Vg.$

It is clear that $V(f\wedge g)\le Vf, Vg$ and assume $Vh\le Vf,Vg.$
By definition of right classes, there are $c_1,c_2\in V$ such that
$c_1\odot h\le f$ and $c_2\odot h \le g.$ Hence, $c_1\le h\to f,$
$c_2\le h\to g$ and $c_1\wedge c_2 \le (h\to f)\wedge (h\to g) = h
\to (f\wedge g)$ giving $(c_1\wedge c_2)\odot h \le f\wedge g.$

Similarly, if  $Vh \ge Vf, Vg,$ there are $c_1,c_2 \in V$ such that
$c_1\odot f \le h$ and $c_2 \odot g \le h$. Then $c_1 \le f\to h$
and $c_2 \le g\to h$ giving $c_1\wedge c_2 \le (f\to h) \wedge (g\to
h)=(f\vee g)\to h$. Whence, $(c_1\wedge c_2)\odot (f\vee g) \le h.$

Now suppose  $Vf\le Vg$ for all values $V$ in $M$ and let $f\not\le
g.$ Then $f\to g <1$ and there is a value $V'$ of $f\to g.$ Then
$V'(f\to g) < V'1 =V'$ and $V'(f\wedge g)= V'((f\to g)\odot f) \le
V'f.$ We note that $V'f\not\le V'(f\wedge g)$ because then $c\odot f
\le (f\wedge g)$ for some $c \in V'$ and $c\le f\to (f\wedge g) =
f\to g$ giving a contradiction $f\to g \in V'.$ By the first part of
the proof, $V'f = V'f\wedge V'g = V'(f\wedge g) <V'f$ that is a
contradiction.

The converse statement is obvious.

The proof of the second statement is the same as that of the first one.
\end{proof}

\section{Visualization}%5

This section will visualize basic pseudo hoops in a Holland's
Representation Theorem type,  see e.g. \cite{Dar} which says that
every $\ell$-group can be embedded into the system of automorphisms
of a linearly ordered set. We show that this result can be extended
also for basic pseudo hoops. We will visualize a basic pseudo hoop
by a system of nondecreasing mapping of a linearly ordered set where
$\odot$-operation corresponds to composition of functions, and the
arrows $\to$ and $\squig$ are defined in a special way.

Let $\Omega$ be a linearly ordered set.  A mapping $f:\Omega \to
\Omega$ is said to be \textit{residutaed} provided there exists a
mapping $f^*:\Omega \to \Omega$ such that $(x)f \le y$ iff $x \le
(y)f^*,$  for all $x,y \in \Omega,$ and we refer to $f^*$ as the
\textit{residual} of $f.$

Let $e=\mbox{id}_\Omega.$ Since $(x)f \le (x)f$ we have $x \le
(x)f\circ f^*$ i.e., $e \le f\circ f^*$ and similarly $f^*\circ f
\le e.$ In addition, $f=f\circ f^*\circ f$ and $f^*=f^*\circ f \circ
f^*.$

If $f_1^*$ and $f^*_2$ are residuals of $f$, then $f_1^*=f_2^*.$
Indeed, we have $f_1^*= f_1^* \circ e \le f_1^*\circ f \circ
f^*_2\le f_2^*$ and by symmetry, $f_1^*= f_2^*.$  Therefore,
$(f\circ g)^*=g^*\circ f^*.$

For example, if $P$ is a prime filter of a basic pseudo hoop, set
$\Omega=M/P$ and given $a \in M$, let $f_a: M/P \to M/P$ be a
mapping defined by $(Px)f_a:= Px\odot a$, $Px \in \Omega_P.$ Then
the residual of $f_a$  is a mapping $f^*_a$ such that $(Px)f_a^*=
P(a\to x),$ $Px \in \Omega.$

Let $\mbox{\rm Mon}(\Omega)$ be the set of all   mappings
$\alpha:\Omega\to \Omega$ such that $\omega_1\le \omega_2$ entails
$(\omega_1)\alpha \le (\omega_2)\alpha.$  We say that $\alpha \le
\beta$ iff $(\omega)\alpha\le (\omega)\beta$ for each $\omega \in
\Omega.$ Then $\mbox{\rm Mon}(\Omega)$ is a lattice ordered
semigroup with the neutral element $e=\mbox{id}_{\Omega}.$

This is the main result of the present section:

\begin{theorem}\label{th:Hol}
Let $M$ be a basic pseudo hoop. Then there is a linearly ordered set
$\Omega$ and a subsystem $\mbox{\rm M}(M)$ of $\mbox{\rm
Mon}(\Omega)$ such that $\mbox{\rm M}(M)$ is a sublattice of
$\mbox{\rm Mon}(\Omega)$ containing $e$ and each element of it is
residuated.  Moreover, $\mbox{\rm M}(M)$ can be converted into a
basic pseudo hoop where the operations are defined pointwise and is
isomorphic to $M$ with the $\odot$-operation corresponding to
composition of functions.
\end{theorem}

\begin{proof}
Let  $\{V_g:g<1\}$ be a system of values, where $V_g$ is a fixed
value of $g<1.$ We define a mapping $\phi_g: M \to \mbox{\rm
Mon}(\Omega_g),$ where $\Omega_g=M/V_g,$ by
$$ (V_gx)\phi_g(a):= V_gx\odot a,\quad
V_gx\in \Omega_g\quad  (a\in M).
$$
Then (i) if $a\le b$, then $\phi_g(a)\le \phi_g(b)$, (ii)
$\phi_g(a)\circ \phi_g(b)=\phi_g(a\odot b),$ (iii) $\phi_g(a\vee
b)=\phi_g(a)\vee \phi_g(b),$ (iv) $\phi_g(a\wedge b)
=\phi_g(a)\wedge \phi_g(b).$ Let $M_0=\prod\{\mbox{\rm
Mon}(\Omega_g): g<1\}$ and order $M_0$ by coordinates. Define a
mapping $f: M\to M_0$ by
$$
f(a)=\{\phi_g(a): g<1\}, \quad a \in M.
$$
By Lemma \ref{le:p4}, $f(a) \le f(b)$ iff $a\le b$ and $f$ is
injective.

Let us totally order the elements of $M\setminus\{1\}$ by $\{g_t: t
\in T\},$ where $T$ is a totally ordered set. Let us set $\Omega_t:=
M/V_{g_t}$ and without loss of generality we can assume $ \Omega_s
\cap \Omega_t = \emptyset$ for all $s,t \in T$ such that $s\ne t$.
Let $\Omega = \bigcup_{t\in T}\Omega_t$, and define a partial order
$\preccurlyeq$ on $\Omega$ by $\omega_1 \preccurlyeq \omega_2$ iff
$\omega_1\in \Omega_s$ and $\omega_2 \in \Omega_t$ and $s <t$ or
$s=t$ and $\omega_1 \le \omega_2$ in $\Omega_s.$ Then $\Omega$ is
totally ordered with respect to $\preccurlyeq$.

Define a mapping $f_0:\ M \to \mbox{Mon}(\Omega)$ by: given $\omega
\in \Omega,$ there is  a unique $t \in T$ such that $\omega \in
\Omega_t.$ Let $(\omega)f_0(a) = (\omega)(\phi_{g_t})(a) \in
\Omega_t.$ Hence, if $a \in M$, then $f_0(a)|_{\Omega_t}$ maps
$\Omega_t$ into $\Omega_t$ for all $t \in T.$ Similarly as for $f$,
$f_0$ is injective and it maps  $M$ onto $\mbox{M}(M):=f_0(M).$ We
have  (i) $f_0(1)=\mbox{id}_\Omega =:e,$ (ii) $f_0(a)\le f_0(b)$ iff
$a\le b$, (iii) $f_0(a)\circ f_0(b) = f_0(a\odot b),$ (iv)
$f_0(a\vee b)= f_0(a)\vee f_0(b),$ (v) $f_0(a\wedge b)= f_0(a)\wedge
f_0(b).$ The residual of $f_0(a),$ $f_0^*(a),$ is defined as
follows: if $\omega \in \Omega_t$ then $\omega = V_{g_t}x$ for some
$x \in M$ and then we set $(\omega)f^*(a) = V_g(x\to a).$

Now we  endow $\mbox{M}(M)$ with the operations: $f_0(a)\odot f_0(b)
:= f_0(a)\circ f_0(b) = f_0(a\odot b)$ and $f_0(a)\to
f_0(b):=f_0(a\to b)$ and $f_0(a)\squig f_0(b):= f_0(a\squig b)$ for
all $a,b \in M.$ Then $\mbox{M}(M)$ is a basic pseudo hoop that is
an isomorphic image of $M$ under the isomorphism $a\mapsto f_0(a),$
$a\in M.$
\end{proof}

\vspace{2mm}

{\bf Question 1.} How we can define $\to$ and $\squig $ in Theorem
\ref{th:Hol} to be defined by points~? We recall that in
\cite{Dvu5}, we have a representation of pseudo MV-algebras by
automorphisms defined on a linearly ordered sets where all
operations, $\odot, \to, \squig$ are defined by points.

\section{Normal-Valued Basic Pseudo Hoops}%6

This is the main part of the this article, where we will study
normal-valued basic pseudo hoops. In particular, we present a
countable system of equations which completely characterize them.

Given $f \in M,$ we define the left and right conjugates,
$\lambda_f$ and $\rho_f,$ of $x\in M$ by $f$ as follows
$$
\lambda_f(x):= f\squig (x\odot f), \quad \rho_f(x):= f\to (f\odot x).
$$
Then a filter $V$ is normal iff $\lambda_f(V) \subseteq V$ and
$\rho_f(V)\subseteq V$ for any $f \in M.$

By \cite[Lem 5.2]{BlTs},
$$
\lambda_f(x\odot y)\le \lambda_f(x) \odot \lambda_f(y),\quad
\rho_f(x\odot y)\le \rho_f(x)\odot \rho_f(y)
$$ for all $x,y \in M.$

Let $V$ be a filter and $f \in M.$ We define
\begin{align*}
f^{-1}Vf&:=\{f\squig (v\odot f): v \in V\}= \lambda_f(V),\\
fVf^{-1}&:=\{f\to (f\odot v): v \in V\}=\rho_f(V).
\end{align*}

Then  a value $V$ of a basic pseudo hoop is normal in $V^*$ iff
$Vf=fV$ for each $f \in V^*$ iff $f^{-1}Vf\subseteq V$ and $fVf^{-1}
\subseteq V$ for each $f \in V^*.$  We say that a basic pseudo-hoop
$M$ is \textit{normal-valued} if every value $V$ of $M$ is normal in
its cover $V^*.$

According to Wolfenstein, \cite[Thm 41.1]{Dar}, an $\ell$-group $G$
is normal-valued iff  every $a,b \in G^-$ satisfy $b^2a^2 \le ab$,
or in our language
$$
b^2\odot a^2 \le a\odot b.\eqno(6.1)
$$
Hence, every cancellative pseudo hoop $M$ is normal-valued iff (6.1)
holds for all $a,b \in M.$  Moreover, every representable pseudo
hoop satisfies (6.1).

Similarly, a pseudo MV-algebra is normal-valued iff (6.1) holds, see
\cite[Thm 6.7]{Dvu2}.

If (6.1) holds in a pseudo hoop $M,$ then given $n\ge 1$ there is an
integer $k_n \ge 1$ such that for all $a,b\in M$
$$
(a\odot b)^n \ge a^{k_n}\odot b^{k_n}.\eqno(6.2)
$$
Indeed, by induction, we have $(a \odot b)^{n+1} = (a\odot b)^n\odot
a\odot b\ge a^{k_n}\odot b^{k_n} \odot a \odot b\ge a^{k_n+2}\odot
b^{2k_n+1}\ge a^{2k_n+2}\odot b^{2k_n+2}.$

If $A,B$ are two subsets of $M,$ we denote by $A\odot B=\{a\odot b:
a\in A,\ b \in B\}.$

\begin{proposition}\label{pr:p5}
Let $M$ be a pseudo hoop.  Then  {\rm (i)} implies {\rm (ii)}, and
{\rm (ii)} and {\rm (iii)} are equivalent, where
\begin{enumerate}

\item[{\rm (i)}] Condition {\rm (6.1)} holds.

\item[{\rm (ii)}] $F(a)\odot F(b) = F(a\odot
b)=F(b\odot a)=F(b) \odot F(a)$ for $a,b \in M.$

\item[{\rm (iii)}] $F\odot G=F\vee G=G\odot F$ for all
filters $F,G \in \mathcal F(M).$

\end{enumerate}

\end{proposition}

\begin{proof} (i) $\Rightarrow$ (ii).   Let $x \in F(a\odot
b).$  There exists $n\ge 1$ and $k_n\ge 1$ such $x \ge (a\odot b)^n
\ge a^{k_n}\odot b^{k_n}.$  (RDP) yields that $x=a_1 \odot b_1$
where $a_1 \ge a^{k_n}$ and $b_1 \ge b^{k_n}$ so that $x = a_1\odot
b_1 \in F(a)\odot F(b).$

Conversely, let $x \in F(a)\odot F(b).$ Then $x = a_1 \odot b_1$ for
some $a_1\in F(a)$ and $b_1\in F(b).$ But then $x \in F(a)\vee
F(b)=F(a\odot b)$ when we have used (4.1). Similarly, $F(b)\odot
F(a)= F(b\odot a).$

(ii) $\Rightarrow$ (iii). It is clear that $F\odot G \subseteq F\vee
G.$  Now take  $x \in F\vee G.$  Then $x \ge a_1\odot b_1 \odot
\cdots \odot a_n\odot b_n$ where $a_i \in F$ and $b_i \in G.$ (RDP)
yields $x = a_1^0\odot b_1^0 \odot \cdots\odot a_n^0\odot b_n^0$ for
$a_i^0 \ge a_i$ and $b^0_i \ge b_i.$  Then $x \in F(a_1)\odot
F(b_1)\odot \cdots \odot F(a_n)\odot F(b_n) = F(a_1)\vee F(b_1)\vee
\cdots \vee F(a_n)\vee F(b_n) = F(a_1)\vee \cdots \vee F(a_n)\vee
F(b_1)\vee \cdots \vee F(b_n) = F(a_1\odot \cdots \odot a_n)\vee
F(b_1\odot \cdots \odot b_n) \subseteq F\vee G.$

(iii) $\Rightarrow$ (ii). We have $F(a)\odot F(b) = F(a)\vee F(b).$
(4.1) entails $ F(a\odot b)=F(a)\vee F(b)=F(b\odot a).$
\end{proof}

\begin{lemma}\label{bot1}
Let $M$ be a basic pseudo hoop. Then, for any $X\subseteq M,$  the
set $X^\perp = \{x :  \ x\vee a=1 \ \forall a\in X\}$ is a filter of
$M.$
\end{lemma}

\begin{proof}
The set $X^\perp$ is clearly closed with respect  to upper bounds.
Let  $x,y\in X^\perp$ and  $a\in X$. The equalities $x\vee a = y\vee
a = 1$ hold. Now we can compute: $(x\odot y) \vee a = (x\odot y)\vee
(x\odot a)\vee a = (x\odot(y\vee a))\vee a = (x\odot 1) \vee a =1$.
Thus also $x\odot y \in X^\perp$.
\end{proof}

\begin{lemma}\label{bot2}
Let $M$ be a basic pseudo hoop with a strong unit $u\in M$. Then the
inclusion

$$
\bigcap \mathrm{Val}(u)\subseteq \{a:  a^n\geq u\ \mbox{\rm for
all}\ n\in \mathbb{N}\}
$$
holds.
\end{lemma}

\begin{proof}
Let $a\in M$ be such an element that there is an integer $n\in
\mathbb{N}$ with the property $a^n \not\geq u$. Thus the inequality
$u\rightarrow a^n< 1$ holds and the filter $\{u\rightarrow
a^n\}^\perp$ is nontrivial (more precisely $u\rightarrow
a^n\not\in\{u\rightarrow a^n\}^\perp$). Prelinearity yields
$a^n\rightarrow u\in \{u\rightarrow a^n\}^\perp$. Because $u$ is a
strong unit and $\{u\rightarrow a^n \}^\perp$ is nontrivial, also
$u\not\in \{u\rightarrow a^n\}^\perp$ holds. Due to Zorn's Lemma,
there is a value $V\in \mathrm{Val }(u)$ such that $\{u\rightarrow
a^n\}^\perp\subseteq V$.

Let us assume to contrary that $a\in V$. Clearly also $a^n,
a^n\rightarrow u\in V$ which gives $(a^n\rightarrow u)\odot a^n \leq
u\in V$ which is a contradiction. Finally, $a\not\in V\supseteq
\bigcap \mathrm{Val} (u)$ and this finishes the proof.
\end{proof}

We recall the following folklore result on prime filters.

\begin{remark}\label{bot3.4}  Let $M$ be a basic pseudo hoop.
Then $$ \bigcap \{F: F\ \mbox{is a minimal prime filter}\}=\{1\}.
$$
\end{remark}

\begin{proof} If $x \in M\setminus\{1\},$ then $\Val(x) \ne
\emptyset$ and any $V \in \Val(x)$ contains a minimal prime filter
$V_M$. This yields $x \notin V \supseteq V_M \supseteq \bigcap \{F:
F$ is a minimal prime filter$\}.$
\end{proof}

\begin{lemma}\label{bot4}
Let $M$ be a basic pseudo hoop and $a,b,x\in M$ be such that
$V(a\odot b)\leq Vx$ for any $V\in\mathrm{Val} (x).$ Then $a^2\odot
b^2\leq x$.
\end{lemma}

\begin{proof}
We are going to prove that $(a^2\odot b^2)\rightarrow x$ belongs to
any minimal prime filter $F$. Let $F$ be a minimal prime filter. If
$x\in F,$ then clearly $(a^2\odot b^2)\rightarrow x\in F$.

We suppose that $x\not\in F$. Thus there exists a value $V\in
\mathrm{Val} (x)$ such that $F\subseteq V$. There are two cases:

(i) Let $a\not\in V$. Clearly, $V(a\odot b^2)\leq V(a\odot b)\leq
Vx$. Hence, $(a\odot b^2)\rightarrow x\in V$ holds. Because
$a\not\in V$ also $((a\odot b^2)\rightarrow x)\rightarrow a \not\in
V$ and, moreover, $((a\odot b^2)\rightarrow x)\rightarrow a \not\in
F$. Prelinearity of $\mathit M$ gives $(a^2\odot b^2)\rightarrow x =
a\rightarrow ((a\odot b^2)\rightarrow x)\in F$.

(ii) Let $a\in V$. We can compute $Vb = V(a\odot b)\leq Vx$ and thus
$b\rightarrow x\in V$. We assert that $b \notin V,$ otherwise, $V1
=V(a\odot b) \le Vx$ yields $x=1\to x \in V,$ which is absurd.
Therefore also $a^2\odot b \not\in V$. Altogether $(b\rightarrow
x)\rightarrow (a^2\odot b)\not\in V$ and consequently $(b\rightarrow
x)\rightarrow (a^2\odot b)\not\in F$. Analogously to the previous
part, prelinearity gives $(a^2\odot b)\rightarrow (b\rightarrow x) =
(a^2\odot b^2)\rightarrow x\in F$.

We have shown that $(a^2\odot b^2)\rightarrow x$ belongs to any
minimal prime filter. Due to Remark \ref{bot3.4}, we obtain
$(a^2\odot b^2)\rightarrow x=1$ and $a^2\odot b^2\leq x$.
\end{proof}

We recall that  a pseudo hoop $M$ is \textit{simple} if it contains
a unique proper filter.

\begin{theorem}\label{th:bot5}
Let $M$ be a  normal-valued basic pseudo hoop, then the following
inequalities hold.
\begin{itemize}
\item[{\rm (i)}] $x^2\odot y^2\leq y\odot x.$
\item[{\rm (ii)}]
$((x\rightarrow y)^n\squig y)^2\leq (x\squig y)^{2n}\to y$ for any
$n\in \mathbb N$.
\item[{\rm (iii)}]
$((x\squig y)^n\to y)^2\leq (x\to y)^{2n}\squig y$ for any $n\in
\mathbb N$.
\end{itemize}
\end{theorem}

\begin{proof}
(i)  For arbitrary $a,b \in M,$  let $x:= b\odot a.$  If $V\in
\mathrm{Val}(x),$ then clearly $a,b\geq x$ yields $a,b\in V^*$.
Because $V^*/V$ is simple (see \cite[Prop 2.3]{DGK}), it is
commutative \cite[Thm 2.4]{DGK}. Then $V(a\odot b) = V(b\odot a)=
Vx$. Due to Lemma \ref{bot4}, we obtain $a^2\odot b^2\leq x = b\odot
a$.

(ii), (iii) For all $x,y\in M$ and each $n\in\mathbb N,$ we denote

$$a := (x\to y)^n\squig y, $$
$$b := (x\squig y)^n,\quad  b':= (x\to y)^n.$$

If $y=1$, (ii) and (iii) trivially hold. Let $y<1$ and let us have
$V\in \mathrm{Val} (y)$. Commutativity of the algebra $V^*/V$ and
$y,x\vee y\in V^*$ yield $V((x\to y)^n) = V(((x\vee y)\to y)^n)=
V(((x\vee y)\squig y)^n) = V((x\squig y)^n)$. Consequently, $Vb=Vb'$
and, moreover, $V(a\odot b) = V(b'\odot a)\leq Vy$. Due to Lemma
\ref{bot4}, we obtain $a^2\odot b^2\leq y$ and also $a^2\leq b^2\to
y$. The second part of the theorem can be proved analogously.
\end{proof}

\begin{lemma}\label{bot6}
Let $M$ be a pseudo hoop  satisfying the inequality $x^2\odot
y^2\leq y\odot x$ and let $F$ be and   $a\in M$ be a fixed filter
and an element of $M$, respectively. Then both sets $\{x\geq f\odot
a^n: n\in\mathbb N, f\in F \}$ and $\{x\geq a^n\odot f: n\in\mathbb
N, f\in F \}$ are equal to the filter generated by $F$ and $a$.
\end{lemma}

\begin{proof}
If $x\geq f_1\odot a^n$  and $y\geq f_2\odot a^m$ are such that
$f_1,f_2\in F$ and $m,n\in \mathbb N$ then $x\odot y\geq f_1\odot
a^n\odot f_2^2\odot a^m\geq f_1\odot f_2\odot a^{2n+m}$. Clearly,
also $f_1\odot f_2^2\in F$ and hence presented sets are filters.
Moreover, the given sets are contained in the filter generated by
$F$ and $a$. This proves the lemma.
\end{proof}

Let us have a  pseudo hoop with inequality $x^2\odot y^2\leq y\odot
x$. If $V$ is a value, then for any $x\in V^*\setminus V,$ we have
$F(V,x)=V^*$ and thus, for any $y\in V^*,$ there are $n\in \mathbb
N$ and $v\in V$ such that $v\odot x^n\leq y$ ($x^n\odot v\leq y,$
respectively). Hence, for any $x\in V\setminus V^*$ and any $y\in
V^*,$ there is $n\in \mathbb N$ such that $V(x^n)\leq Vy$ (or
$(x^n)V\leq yV$).

\begin{theorem}\label{th:bot6}
If a basic pseudo hoop $M$ satisfies inequalities
{\rm (i)--(iii)} from Theorem {\rm\ref{th:bot5}}, then $M$ is
normal-valued.
\end{theorem}

\begin{proof}
Let (i)--(iii) hold and let $V$ be a value. Let  $x,y\in V^*$ be such
that $x\to y\not\in V$ (and hence $y\not\in V$). Then there is
$n\in\mathbb N$ such that $V(x\to y)^n\leq Vy$. Hence, $(x\to
y)^n\to y\in V$ and also $((x\to y)^n\to y)^2\in V$. Due to
inequality (ii),  $(x\squig y)^{2n}\to y\in V$ holds. Hence, we
assert $x\squig y \notin V.$ If not, $x\squig y\in V$ yields $y\geq
((x\squig y)^{2n}\to y)\odot (x\squig y)^{2n}\in V$ which is a
contradiction. Altogether $x\to y\not\in V$ yields $x\squig y\not\in
V.$

The converse implication $x\squig  y\not\in V$ yields $x\to y\not\in
V$ can be proved in an analogous way. This implies $M$ is normal-valued.
\end{proof}

Combining the results of Theorem \ref{th:bot6} and Theorem
\ref{th:bot5}, we have the following corollary.

\begin{corollary}\label{bot:cor}
Let $M$ be a basic pseudo hoop. The following
statements are equivalent

\begin{enumerate}
\item[{\rm (i)}] $M$ is  normal-valued.

\item[{\rm (ii)}] {\rm (i)--(iii)} from Theorem {\rm\ref{th:bot5}} hold.
\end{enumerate}

\end{corollary}

\vspace{2mm}

\begin{lemma}\label{bot7}
If a basic pseudo hoop $M$ satisfies the inequality $x^2\odot
y^2\leq y\odot x,$ then any value $V$ and any $x\in V^*\setminus V$
such that $Vx>V(x^2)$ satisfy $Vx\subseteq xV$.
\end{lemma}

\begin{proof}
Let us have $x\in V^*$ such that $Vx>Vx^2$ and, moreover, let $f\in
V$ be such that $\lambda_x(f)=x\squig (f\odot x)\not\in V$. The
divisibility clearly yields the equality $
x\odot(\lambda_x(f))^n=f^n\odot x$ for any $n\in\mathbb N$. Because
$f\in V$, we can interpret the last equality as
$V((\lambda_x(f))^n)\geq V(x\odot(\lambda_x(f))^n) = V(f^n\odot x)
=Vx$. Hence, $\lambda_x(f)\in V^*\setminus V$ and $x\in V^*.$ There
is $n\in\mathbb N$ such that $Vx^2\geq V((\lambda_x(f))^n)$ and
altogether $Vx^2< Vx$ which is  a contradiction. We have proved that
for any $f\in V$ also $\lambda_x(f)\in V$.

One can easily check that for any $y\in Vx,$ the inequality
$Vy^2<Vy$ holds (more precisely, if $Vy^2=Vy,$ then $Vy$ is a least
element and also $Vx=Vy$ is minimal which gives a contradiction
$Vx=Vx^2$). Due to $y\in Vx,$ we obtain the equality $f_1\odot
x=f_2\odot y$ and thus also $x\odot \lambda_x(f_1)=y\odot
\lambda_y(f_2)$. In the previous part, we have proved that
$\lambda_x(f_1),\lambda_y(f_2)\in V$ and thus $y\in xV$.
\end{proof}

{\bf Question 2.}  Does inequality $x^2\odot y^2\leq y\odot x$
characterize the class of (basic) normal-valued pseudo hoops~?  For
example, let $G$ be an $\ell$-group and let $G^\dag$ be the kite
corresponding to $G$ (for kites see \cite{JiMo, DGK}).  By \cite[Lem
4.11]{DGK}, the kite $G^\dag$ is a normal-valued pseudo BL-algebra
iff $G$ is a normal-valued $\ell$-group. Hence, inequality (6.1)
completely characterizes a kite to be normal-valued.

In what follows, we present a variety of basic pseudo hoops
satisfying a single equation such that the inequality $x^2\odot y^2
\le y\odot x$ is a necessary and sufficient condition for  $M$ to be
normal-valued.

We say that a bounded pseudo hoop $M$ is \textit{good} if
$$
x^{-\sim}=x^{\sim-}, \quad x \in M, \eqno(6.3)
$$
where $x^-:=x\to 0$ and $x^\sim=x\squig 0.$  For example, every
pseudo MV-algebra is good as well as every representable pseudo hoop
is good, see \cite{Dvu4}. On the other hand, a kite $G^\dag$ is a
pseudo BL-algebra which is not good whenever $G\ne \{e\},$ \cite[Lem
4.11]{DGK}.

We present a stronger equality than (6.3):
$$
(x\to y)\squig y = (x\squig y)\to y \eqno(6.4)
$$
for all $x,y \in M.$

For example, every negative cone of an $\ell$-group and the negative interval of an $\ell$-group with strong unit satisfies (6.4).
If  $M$ is a linearly ordered pseudo hoop, due to
\cite[Cor 4.2]{Dvu4}, $M$ is an ordinal sum of a system whose each
component is either the negative cone of a linearly ordered
$\ell$-group or the negative interval of a linearly ordered
$\ell$-group with strong unit.  Therefore, it satisfies (6.4),
consequently every representable  bounded pseudo hoop satisfies (6.4). On the
other side, no nontrivial kite satisfies (6.4).

\begin{lemma}\label{tom1}  Let $M$ be a basic pseudo hoop satisfying
{\rm (6.4)}. Then $M$ satisfies the identity
$$ (x\to y)^n \squig y= (x\squig y)^n \to y \eqno(6.5)
$$
for all $x,y \in M$ and for any $n \in \mathbb N.$
\end{lemma}

\begin{proof}
Assume for induction that (6.5) holds for any integer $k$ with $1\le
k\le n.$ We have
\begin{align*}
(x\to y)^{n+1} \squig y &= (x\to y)^n \squig ((x\to y)\squig y)\\
&= (x\to y)^n \squig ((x\squig y)\to y)\\
&= (x \squig y) \to ((x\to y)^n \squig y)\\
&= (x\squig y) \to ((x\squig y)^n \to y\\
&= (x\squig y)^{n+1}\to y,
\end{align*}
where in the third equality we have used the identity $a \squig (b
\to c)= b \to (a\squig c),$ $a,b,c \in M,$ see \cite[Lem
3.2(6)]{BlTs}.
\end{proof}

\begin{theorem}\label{tom2}
Let $M$ be a basic pseudo hoop satisfying {\rm (6.4)}.  Then $M$ is
normal-valued if and only if $x^2 \odot y^2 \le y\odot x$ for all
$x,y \in M.$
\end{theorem}

\begin{proof}
The ``if" condition holds by Theorem \ref{th:bot5}, so only need to
show that if $M$ satisfies $x^2 \odot y^2 \le y\odot x,$ then $M$ is
normal-valued. Assume $M$ satisfies $x^2 \odot y^2 \le y\odot x.$ By
Theorem \ref{th:bot6}, it suffices to show that $M$ satisfies
$$
((x\rightarrow y)^n\squig y)^2\leq (x\squig y)^{2n}\to y
$$
and
$$
((x\squig y)^n\to y)^2\leq (x\to y)^{2n}\squig y
$$
for any $n\in \mathbb N$.  By Lemma \ref{tom1}, it is enough to show
that the first inequality of the latter two hods, which is, by
residuation, equivalent to

$$
((x\rightarrow y)^n\squig y)^2\odot (x\squig y)^{2n}\le y.
$$
Now, consider
\begin{align*}
((x\rightarrow y)^n\squig y)^2\odot (x\squig y)^{2n} &=
((x\rightarrow y)^n\squig y)\odot ((x\rightarrow y)^n\squig y)\odot
(x\squig y)^{2n}\\
&= ((x\rightarrow y)^n\squig y)\odot ((x\squig y)^n\to y)\odot
(x\squig y)^{2n} \\
&\le ((x\rightarrow y)^n\squig y)\odot y \odot (x\squig y)^{n}\\
&\le y
\end{align*}
showing that the desired inequality holds.
\end{proof}

The following statement was proved in \cite[Thm 3.2]{DGK}
in a different way, here we use Theorem \ref{tom2}.

\begin{corollary}\label{tom3}
Every representable pseudo hoop is normal-valued.
\end{corollary}

\begin{proof}  If $M$ is a linearly ordered pseudo hoop, then
according to the remark just after (6.4), $M$ satisfies $(6.4).$
Every linearly ordered $\ell$-group is normal-valued \cite{Dar}, so
is its negative cone as well as  its negative interval  with strong
unit satisfies the inequality $x^2\odot x^2\le y \odot x.$
Consequently, every ordinal sum of such linear components satisfies
the inequality which  by Theorem \ref{tom2} entails, $M$ is
normal-valued.
\end{proof}

\end{document}